\documentclass[reqno]{amsart}
\usepackage{amsmath}
\usepackage{amsfonts}
\usepackage{amssymb}
\usepackage{graphicx}
\usepackage{amsthm}
\usepackage{latexsym}
\usepackage{setspace}

\newtheorem{theorem}{Theorem}[section]

\newtheorem{proposition}[theorem]{Proposition}
\newtheorem{corollary}[theorem]{Corollary}
\newtheorem{lemma}[theorem]{Lemma}

\theoremstyle{definition}

\newtheorem{remark}[theorem]{Remark}
\newtheorem{example}[theorem]{Example}

\begin{document}
\title[Exponential stability]{A complete characterization of exponential
stability for discrete dynamics}
\author[N. Lupa]{Nicolae Lupa}
\address{N. Lupa, Department of Mathematics, Politehnica University of
Timi\c soara, Victoriei Square 2, 300006 Timi\c soara, Romania}
\email{nicolae.lupa@upt.ro}
\author[L.H. Popescu]{Liviu Horia Popescu}
\address{L. H. Popescu, Department of Mathematics and Informatics, Faculty
of Sciences, University of Oradea, Universit\u a\c tii St. 1, 410087 Oradea,
Romania}
\email{lpopescu2002@yahoo.com}

\begin{abstract}
For a discrete dynamics defined by a sequence of bounded and
not necessarily invertible linear operators, we give a
complete characterization of exponential stability in terms of invertibility
of a certain operator acting on suitable Banach sequence spaces. We connect
the invertibility of this operator to the existence of a particular type of
admissible exponents. For the bounded orbits, exponential stability results
from a spectral property. Some adequate examples are presented to
emphasize some significant qualitative differences between uniform and
nonuniform behavior.
\end{abstract}

\maketitle

\section{Introduction}

Throughout our work we deal with nonautonomous linear difference equation
\begin{equation}
x_{n+1}=A_{n}x_{n},\;n\in\mathbb{N},   \label{eq.dif}
\end{equation}
where $\left( A_{n}\right) _{n\in\mathbb{N}}$ is a sequence of bounded
and not necessarily linear operators acting on a Banach space $E$.

There is an analogy between difference equations and differential equations,
via discretization method. To be more specific, a discrete time equation in
an infinite dimensional space typically occurs as a discretization of a
compact semigroup (cf. \cite[pp. 12]{Po.2010}). It is therefore natural to
consider the case when the operators $A_{n}$ are not necessarily invertible.
Furthermore, it is sometimes possible to apply some results in the discrete
case to the continuous situation. More precisely, in order to study the asymptotic behavior of a nonautonomous dynamical system modeled by an evolution family it suffices to characterize the asymptotic behavior of the associated
discrete dynamics.  On this matter we refer to \cite[Discrete
Dichotomy Theorem]{Ch.La.1999}, \cite[pp. 229-230, Exercise 10]{He}, \cite[Theorem 4.8]{Hu.Mi.2001}
and \cite[Lemma 3.1]{Sa.Sa} for the case of uniform exponential dichotomies, as well as
to \cite[Theorem 4.3]{Sa.Sa.tr} for exponential trichotomies. For useful and interesting
details on the theory of discrete dynamical systems, we refer the reader to
monographs \cite{Ag,El,Po.2010}. Some important results for discrete
dynamics are also obtained in \cite{Au.Mi,Na.Pi,Pham.Na,Pi}.

In the study of the asymptotic behavior of solutions of difference
equations, it is important to find conditions for the existence of
exponential stability, dichotomy or trichotomy of the associated dynamics.
In this regard one possible approach is to reduce the study of the qualitative behaviour of the
nonautonomous equation \eqref{eq.dif} to the analysis of the spectral properties of the operator $\mathbb{A}$ defined by
$$(\mathbb{A}x)_{n+1}=A_{n}x_{n},\;n\in\mathbb{N},\; x=(x_n)_{n\in\mathbb{N}}\in Y,$$
where $Y$ is some Banach sequence space with elements in $E$. Thus, this approach consists of passing from a nonautonomous dynamics to
an autonomous one defined on a
class of \textquotedblleft super\textquotedblright-spaces of $E$-valued
sequences, which is simpler to study (see for instance \cite{Bu,Ch.La.1999,Hu.Mi.2001}).
This idea goes back to the pioneering works of Perron~\cite{Pe.1930} in the case of ordinary differential equations and of Li~\cite{Li.1934} for discrete time systems.
A related tool
is represented by the input-output techniques, often called \emph{%
admissibility methods}. Such arguments involve the existence of some pairs
of Banach sequence spaces, and replace the study of the equation %
\eqref{eq.dif} to the analysis of the control system
\begin{equation*}
u_{n+1}=A_{n}u_{n}+f_{n+1},\;n\in\mathbb{N}.
\end{equation*}
For example in \cite{Sa.Sa} it is shown that the admissibility of the pair $%
\left( \ell^{p}(\mathbb{N},E),\ell_{0}^{q}(\mathbb{N},E)\right) $, $1\leq
p,\,q<\infty$, where
$$\ell_{0}^{q}(\mathbb{N},E)=\{x=(x_{n})_{n\in\mathbb{N}}\in\ell^{q}(\mathbb{N%
},E):\,x_{0}=0\},$$ implies the existence of uniform exponential dichotomy.

In recent years, a large number of papers study different  aspects of the dynamics of systems
with nonuniform exponential behavior.  A
motivation for introducing this type of behavior lies in the ergodic theory, almost all trajectories with nonzero Lyapunov exponents have a nonuniform exponential behavior. We refer the reader to the work of Barreira and Valls \cite{Ba.Va.2008-1} for details and references on the theory of
nonuniform exponential stability.
In \cite{Ba.Va.2016} the authors  showed that the
admissibility of the pair $\left( \ell_{\alpha}^{\infty}(\mathbb{N}%
,E),\ell_{\alpha}^{\infty}(\mathbb{N},E)\right) $ suffices for a discrete
dynamics to be nonuniform exponentially stable. Here $\ell_{\alpha}^{\infty
}(\mathbb{N},E)$ denotes the following Banach sequence space:
\begin{equation*}
\ell_{\alpha}^{\infty}(\mathbb{N},E)=\left\{ x=(x_{n})_{n\in\mathbb{N}%
}\subset E:\;\sup\limits_{n\in\mathbb{N}}\Vert
x_{n}\Vert_{n,\alpha}<\infty\right\} ,
\end{equation*}
where
\begin{equation*}
\Vert v\Vert_{n,\alpha}=\sup\limits_{m\geq n}e^{-\alpha(m-n)}\Vert
A(m,n)v\Vert,\;v\in E,
\end{equation*}
$n\in\mathbb{N}$ and $\alpha\in\mathbb{R}$ is such that there exist some
constants $M>0$ and $\varepsilon\geq0$ with $\Vert A(m,k)\Vert\leq
Me^{\varepsilon k}e^{\alpha(m-k)}$, for all $m\geq k$ in $\mathbb{N}$. For
related results in the case of nonuniform exponential dichotomy we refer the
reader to \cite{Ba.Dav}. Furthermore, in \cite{Ba.Dav.Va} the authors completely characterize the general notion of \emph{exponential dichotomy with respect to a sequence of norms} $\|\cdot\|_n$, $n\in\mathbb{Z}$, such that each norm  $\|\cdot\|_n$ is equivalent to the original norm on a Banach space $E$, in terms terms of  the admissibility of some pairs $(Y_1,Y_2)$ obtained from $(l^p(\mathbb{Z},E),l^q(\mathbb{Z},E))$, with $p$ and $q$ not necessarily equal.
Other recent important results in the theory of
nonuniform behavior for discrete dynamics are obtained in \cite%
{Be.Si,Chu,Po.Ce.Me,Po.Me.Ce,Zh}.

The main goal of our paper is to present a complete description of
nonuniform exponential stability of the discrete dynamics defined by Eq. %
\eqref{eq.dif}, in the spirit of the first method presented above. Let us
notice that the techniques used in the uniform case do not work in the
nonuniform setting, therefore new approaches must be identified. In this
framework, our constructions and arguments (needed in the case of nonuniform
behavior) are nontrivial.
To the best of our knowledge, the statement in
Theorem \ref{t1} gives a new result in the stability theory, and evidently
it is specific to nonuniform behavior.
A sufficient condition for the
existence of nonuniform exponential stability of all bounded orbits is also
established. Important, some discrete versions of well known results in the
stability theory of evolution equations are simple consequences of our
conclusions. 

Similar results for nonuniform hyperbolic sequences are obtained in the recent paper \cite{Ba.Po.Va}%
, that treats the case when the sequence $\mathcal{A}$ is $\mathbb{Z}$%
-indexed. The main result in \cite{Ba.Po.Va} tells that $\mathcal{A}=(A_n)_{n\in\mathbb{Z}}$ is hyperbolic with respect to an appropriate family of norms $\|\cdot\|_n$ if and only if the associated evolution map is hyperbolic.  When lecturing \cite{Ba.Po.Va} one may easily observe that the
path we follow is of a different type,  in \cite{Ba.Po.Va} an appropriate sequence of norms is used instead of the concept of admissible exponents. Let us also
point out that our paper
proves the necessity of the study of each individual admissible exponent
situation, as illustrated in Example \ref{ex.2}.

The results obtained in this paper are also related to the results obtained by Barreira and Valls \cite{Ba.Va.2016} and by Barreira, Dragi\v{c}evi\'{c} and Valls \cite{Ba.Dav.Va,Ba.Dav}, respectively in the case of nonuniform exponential stability and nonuniform hyperbolicity. In particular, if we consider the projections $P_n$ in \cite{Ba.Dav.Va,Ba.Dav} to be the identity, we have the  case of nonuniform exponential  stability. In this particular situation, Theorem 1 and Theorem  2 in \cite{Ba.Dav}, respectively  Theorem 4.1 and Theorem  5.2 in \cite{Ba.Dav.Va} are strongly related to Theorem \ref{th1} in the present paper. In spite of this, the characterizations are different, since  none of the results in \cite{Ba.Dav.Va,Ba.Dav,Ba.Va.2016} and in our paper imply the results in the other.

\section{Notions and preliminary results}

\label{sec.2}

As usually $\mathbb{N}$ denotes the set of all non-negative integers, and we
put $\mathbb{N}^{\ast}=\mathbb{N}\setminus\{0\}$. Let $\Delta=\{(m,n)\in
\mathbb{N}\times\mathbb{N}:\,m\geq n\}$ and $\Delta_{n}=\{m\in\mathbb{N}%
:\,(m,n)\in\Delta\}$, $n\in\mathbb{N}$. Given a sequence $\mathcal{A}=\left(
A_{n}\right) _{n\in\mathbb{N}}$ of bounded linear operators on a Banach
space $E$, for any $(m,n)\in\Delta$ we set
\begin{equation*}
A(m,n)=%
\begin{cases}
A_{m-1}\cdots A_{n}, & m>n, \\
\mathrm{Id}, & m=n.%
\end{cases}
\end{equation*}
The family of operators $\{A(m,n)\}_{(m,n)\in\Delta}$ is called the \emph{%
discrete evolution family} associated to $\mathcal{A}$. Notice that we do
not assume the invertibility condition for the operators $A_{n}$, thus our
results are applicable to a significantly large class of discrete dynamics.

For fixed $\alpha\in\mathbb{R}$ the operator sequence $\mathcal{A}$ is
called \emph{$\alpha$-nonuniform exponentially bounded,} if there exists a
(not necessarily bounded) positive real sequence $(M_{n})_{n\in\mathbb{N}}$
such that%
\begin{equation}
\left\Vert A(m,n)\right\Vert \leq M_{n}e^{\alpha(m-n)},\text{ for }%
(m,n)\in\Delta,   \label{eq1}
\end{equation}
in which case we call $\alpha$ an \emph{admissible exponent} for $\mathcal{A}
$. We denote $\mathcal{E}_{ad}\left( \mathcal{A}\right) $ the set of all
admissible exponents of $\mathcal{A}$. If there exists $\alpha\in \mathcal{E}%
_{ad}\left( \mathcal{A}\right) $ with $\alpha<0$, then $\mathcal{A}$ is
called \emph{$\alpha$-nonuniform exponentially stable}.

If $\mathcal{E}_{ad}\left( \mathcal{A}\right) \neq\emptyset$ we say that $%
\mathcal{A}$ is \emph{nonuniform exponentially bounded}, and if $\mathcal{A}$
has negative admissible exponents, then $\mathcal{A}$ is called \emph{%
nonuniform exponentially stable}.

In the above terminology, whenever there exists a \emph{bounded} sequence $%
(M_{n})_{n\in\mathbb{N}}$ satisfying \eqref{eq1}, we just replace the term
\textquotedblleft nonuniform\textquotedblright\ by \textquotedblleft
uniform\textquotedblright.

We specify that throughout our paper we always assume $\mathcal{E}%
_{ad}\left( \mathcal{A}\right) \neq\emptyset$. For any fixed $\alpha\in%
\mathcal{E}_{ad}\left( \mathcal{A}\right) $ and $n\in\mathbb{N}$, the below
mapping \cite[Eq. (6)]{Ba.Va.2016} is a norm on $E$:
\begin{equation*}
\Vert x\Vert_{n,\alpha}=\sup\limits_{m\in\Delta_{n}}e^{-\alpha(m-n)}\Vert
A(m,n)x\Vert,\;x\in E,
\end{equation*}
that satisfies
\begin{equation*}
\| x\|\leq \|x\|_{n,\alpha}\leq M_n\|x\|.
\end{equation*}

Let us denote
\begin{equation*}
c_{00}=\left\{ u=\left( u_{n}\right) _{n\in\mathbb{N}}\subset
E:\,\lim\limits_{n\rightarrow\infty}u_{n}=u_{0}=0\right\} ,
\end{equation*}
which is a Banach sequence space equipped with norm $\left\Vert u\right\Vert
_{\infty}=\sup\limits_{n\in\mathbb{N}}\Vert u_{n}\Vert$. For each $\alpha \in%
\mathcal{E}_{ad}\left( \mathcal{A}\right) $ we also denote%
\begin{equation*}
c_{00}(\alpha)=\left\{ u=(u_{n})_{n\in\mathbb{N}}\subset E\text{: }%
\lim\limits_{n\rightarrow\infty}\Vert u_{n}\Vert_{n,\alpha}=\Vert
u_{0}\Vert=0\right\} .
\end{equation*}
Standard arguments can show that $c_{00}(\alpha)$ is a Banach
sequence space endowed with norm
\begin{equation*}
\Vert u\Vert_{\alpha}=\sup\limits_{n\in\mathbb{N}}\Vert u_{n}\Vert_{n,\alpha
}.
\end{equation*}

\begin{remark}
\label{r1} We have

\begin{enumerate}
\item[(i)] $c_{00}(\alpha)\subset c_{00}$, $\alpha\in\mathcal{E}_{ad}\left(
\mathcal{A}\right) $;

\item[(ii)] if $\alpha\in\mathcal{E}_{ad}\left( \mathcal{A}\right) $ and $%
\beta\geq\alpha$, then $\beta\in\mathcal{E}_{ad}\left( \mathcal{A}\right) $
and $c_{00}(\alpha)\subset c_{00}(\beta)$.
\end{enumerate}
\end{remark}

For a nonautonomous dynamics $\mathcal{A}=\left( A_{n}\right) _{n\in \mathbb{%
N}}$, the \emph{evolution map} $T$ is defined by
\begin{equation*}
(Tu)_{n}=
\begin{cases}
A_{n-1}u_{n-1}, & n\geq1, \\
0, & n=0,%
\end{cases}
\end{equation*}
on the set of all $E$-valued sequences $u=(u_{n})_{n\in\mathbb{N}}$. For
each $\alpha\in\mathcal{E}_{ad}\left( \mathcal{A}\right) $, let $T_{\alpha}$
be the restriction of the operator $T$ to $c_{00}(\alpha)$.

\begin{proposition}
\label{lem1}The map $T_{\alpha}:c_{00}(\alpha)\rightarrow c_{00}(\alpha)$ is
a bounded linear operator on $c_{00}(\alpha)$.
\end{proposition}

\begin{proof}
Let $u\in c_{00}(\alpha)$. For every positive integer $n\in\mathbb{N}^{\ast}$
one has
\begin{align*}
\Vert(T_{\alpha}u)_{n}\Vert_{n,\alpha} &
=\sup\limits_{m\in\Delta_{n}}e^{-\alpha(m-n)}\Vert
A(m,n)(T_{\alpha}u)_{n}\Vert \\
& =\sup\limits_{m\in\Delta_{n}}e^{-\alpha(m-n)}\Vert A(m,n-1)u_{n-1}\Vert \\
& \leq e^{\alpha}\Vert u_{n-1}\Vert_{n-1,\alpha}.
\end{align*}
The above estimation immediately yields $T_{\alpha}u\in c_{00}(\alpha)$,
thus $T_{\alpha}:c_{00}(\alpha)\rightarrow c_{00}(\alpha)$ is well defined.
We also deduce that $\Vert T_{\alpha}u\Vert_{\alpha}\leq e^{\alpha}\Vert
u\Vert_{\alpha}$, consequently the linear operator $T_{\alpha}$ is bounded
on $c_{00}(\alpha)$, with $\left\Vert T_{\alpha}\right\Vert
_{c_{00}(\alpha)}\leq e^{\alpha}$.
\end{proof}

Let us put
\begin{equation*}
G_{\alpha}=T_{\alpha}-\mathrm{Id},\;\alpha\in\mathcal{E}_{ad}\left( \mathcal{%
A}\right) ,
\end{equation*}
that is $G_{\alpha}$ is the restriction of $G=T-\mathrm{Id}$ to $%
c_{00}(\alpha)$. We remark that the operator $G_{\alpha}$ plays in fact the
role of the infinitesimal generator of the evolution semigroup (see \cite[%
pp. 73]{Ch.La.1999} and \cite{Mi.Ra.Sc.1998}), in the case of nonuniformly
exponentially bounded discrete dynamics.

Next lemma is crucial in the sequel, and will be used latter to prove our
first main result.

\begin{lemma}
\label{lem2-1}Let $\alpha\in\mathcal{E}_{ad}(\mathcal{A})$.

\begin{enumerate}
\item[(i)] If $u,v\in c_{00}(\alpha)$, then $G_{\alpha}u=-v$ if and only if
the following identity holds:
\begin{equation}
u_{n}=\sum\limits_{k=0}^{n}A(n,k)v_{k},\;n\in\mathbb{N}.   \label{eq3}
\end{equation}

\item[(ii)] If $\alpha<0$, then for each $v\in c_{00}(\alpha)$, let $%
u=(u_{n})_{n\in\mathbb{N}}$ be the sequence defined by Eq. \eqref{eq3}. Then
$u\in c_{00}(\alpha)$ and $G_{\alpha}u=-v$.
\end{enumerate}
\end{lemma}

\begin{proof}
(i) \emph{Necessity}. Assume that $u,v\in c_{00}(\alpha)$ with $G_{\alpha
}u=-v$. Evidently \eqref{eq3} holds for $n=0$. If $n\in\mathbb{N}^{\ast}$ we
successively have
\begin{equation*}
T_{\alpha}^{n}u-u=\sum\limits_{k=0}^{n-1}\left( T_{\alpha}^{k+1}-T_{\alpha
}^{k}\right) u=\sum\limits_{k=0}^{n-1}T_{\alpha}^{k}G_{\alpha}u=-\sum
\limits_{k=0}^{n-1}T_{\alpha}^{k}v,
\end{equation*}
therefore
\begin{equation*}
u_{n}=\sum\limits_{k=0}^{n}(T_{\alpha}^{k}v)_{n}=\sum%
\limits_{k=0}^{n}A(n,n-k)v_{n-k}=\sum\limits_{k=0}^{n}A(n,k)v_{k}.
\end{equation*}
\emph{Sufficiency}. If \eqref{eq3} holds, then for each $n\in\mathbb{N}%
^{\ast }$ one has
\begin{align*}
(G_{\alpha}u)_{n} & =A_{n-1}u_{n-1}-u_{n} \\
&
=A(n,n-1)\sum\limits_{k=0}^{n-1}A(n-1,k)v_{k}-\sum%
\limits_{k=0}^{n}A(n,k)v_{k} \\
&
=\sum\limits_{k=0}^{n-1}A(n,k)v_{k}-\sum\limits_{k=0}^{n}A(n,k)v_{k}=-v_{n}.
\end{align*}
As the above relation also holds for $n=0$, we conclude that $G_{\alpha}u=-v$%
.

(ii) Choose $v\in c_{00}(\alpha)$ and consider $u=(u_{n})_{n\in\mathbb{N}}$
defined by Eq. \eqref{eq3}. For $n\in\mathbb{N}$, using \eqref{eq3}, one
gets
\begin{align*}
\Vert u_{n}\Vert_{n,\alpha} & =\sup\limits_{m\in\Delta_{n}}e^{-\alpha
(m-n)}\Vert A(m,n)u_{n}\Vert \\
& \leq\sup\limits_{m\in\Delta_{n}}e^{-\alpha(m-n)}\sum\limits_{k=0}^{n}\Vert
A(m,k)v_{k}\Vert \\
&= \sum\limits_{k=0}^{n}e^{\alpha(n-k)}\sup\limits_{m\in\Delta_{n}}e^{-%
\alpha(m-k)}\Vert A(m,k)v_{k}\Vert \\
&\leq\sum\limits_{k=0}^{n}e^{\alpha(n-k)}\sup\limits_{m\in\Delta_{k}}e^{-%
\alpha(m-k)}\Vert A(m,k)v_{k}\Vert \\
& =\sum\limits_{k=0}^{n}e^{\alpha(n-k)}\Vert v_{k}\Vert_{k,\alpha}.
\end{align*}
As for $\alpha<0$ the series $\sum\limits_{n=0}^{\infty}e^{\alpha n}$ is
convergent and $\lim\limits_{n\rightarrow\infty}\Vert v_{n}\Vert_{n,\alpha}=0
$, Lemma \ref{l1.A} readily implies $\lim\limits_{n\rightarrow\infty}\Vert
u_{n}\Vert_{n,\alpha}=0$, thus $u\in c_{00}(\alpha)$. The second statement
is a direct consequence of (i).
\end{proof}

\section{Nonuniform exponential stability and invertibility of the operators
$G_{\protect\alpha}$}

\label{sec.3}

In this section we completely characterize the existence of nonuniform
exponential stability for a discrete dynamics $\mathcal{A}$, in terms of
invertibility of the operators $G_{\alpha}$.

We consider next theorem to be one of the main results of the paper. It is
in fact the correspondent of Theorem 2.2 in \cite{Mi.Ra.Sc.1998}, for a
nonuniform exponentially bounded discrete dynamics. Let us emphasize that
the time-dependence of the intervening constants requires completely
different techniques.

\begin{theorem}
\label{th1}The sequence $\mathcal{A}$ is nonuniform exponentially stable if
and only if there exists an admissible exponent $\alpha\in\mathcal{E}_{ad}(%
\mathcal{A})$, for which the operator $G_{\alpha}$ is invertible on $%
c_{00}(\alpha)$.
\end{theorem}

\begin{proof}
Necessity. Assume that $\mathcal{A}$ is $\alpha$-nonuniform exponentially
stable. Lemma \ref{lem2-1} (i) readily implies that $G_{\alpha}$ is a
one-to-one operator, then as $\alpha<0$, Lemma \ref{lem2-1} (ii) ensures
that $G_{\alpha}:c_{00}(\alpha)\rightarrow c_{00}(\alpha)$ is onto.

Sufficiency. If the operator $G_{\alpha}$ is invertible on $c_{00}(\alpha)$
for some $\alpha\in\mathcal{E}_{ad}(\mathcal{A})$, then evidently it
suffices to analyze only the case $\alpha\geq0$. As $G_{\alpha}^{-1}$ is
bounded on $c_{00}(\alpha)$, let us denote $c_{\alpha}=\Vert
G_{\alpha}^{-1}\Vert>0$. We divide our proof into three steps.

\textbf{Step 1}. The following inequality holds:
\begin{equation}
\Vert A(m,n)\Vert\leq\left[ c_{\alpha}(e^{\alpha}-1)+1\right] M_{n},\text{
for }(m,n)\in\Delta.   \label{eq4}
\end{equation}

For $\alpha=0$, estimation \eqref{eq4} is a simple consequence of Eq. %
\eqref{eq1}. Pick $\alpha\in\mathcal{E}_{ad}(\mathcal{A})$ with $\alpha>0$.
If $m>n$ are fixed, arbitrary non-negative integers and $x\in E$, we define
the sequence $v=(v_{k})_{k\in\mathbb{N}}$ (depending on $\alpha$, $m$, $n$, $%
x$) by
\begin{equation*}
v_{k}=%
\begin{cases}
e^{-\alpha(k-n)}A(k,n)x, & k\in\{n+1,\ldots,m\}, \\
0, & \text{otherwise},%
\end{cases}
\end{equation*}
and if $m=n$, we put $v_{k}=0$ for all $k\in\mathbb{N}$. Let us remark that $%
v\in c_{00}(\alpha)$. If $k\notin\{n+1,\ldots,m\}$, then $\Vert
v_{k}\Vert_{k,\alpha}=0$. For $k\in\{n+1,\ldots,m\}$ we have
\begin{equation*}
\Vert
v_{k}\Vert_{k,\alpha}=\sup\limits_{i\in\Delta_{k}}e^{-\alpha(i-k)}\Vert
A(i,k)v_{k}\Vert=\sup\limits_{i\in\Delta_{k}}e^{-\alpha(i-n)}\Vert
A(i,n)x\Vert\leq M_{n}\Vert x\Vert,
\end{equation*}
consequently
\begin{equation*}
\label{eq.prel}\Vert v\Vert_{\alpha}\leq M_{n}\Vert x\Vert.
\end{equation*}
Putting $u=G_{\alpha}^{-1}(-v)\in c_{00}(\alpha)$, from identity \eqref{eq3}
one gets
\begin{equation*}
u_{m}=\sum\limits_{k=0}^{m}A(m,k)v_{k}=\sum\limits_{k=n+1}^{m}e^{-\alpha
(k-n)}A(m,n)x=\dfrac{1-e^{-\alpha(m-n)}}{e^{\alpha}-1}\,A(m,n)x,
\end{equation*}
which implies
\begin{equation*}
\dfrac{1}{e^{\alpha}-1}A(m,n)x=u_{m}+\dfrac{1}{e^{\alpha}-1}\,e^{-\alpha
(m-n)}A(m,n)x.
\end{equation*}
This yields
\begin{align*}
\dfrac{1}{e^{\alpha}-1}\Vert A(m,n)x\Vert & \leq\Vert u_{m}\Vert+\dfrac {1}{%
e^{\alpha}-1}\,e^{-\alpha(m-n)}\Vert A(m,n)x\Vert \\
& \leq\Vert u\Vert_{\alpha}+\dfrac{1}{e^{\alpha}-1}M_{n}\Vert x\Vert \\
& \leq c_{\alpha}\Vert v\Vert_{\alpha}+\dfrac{1}{e^{\alpha}-1}M_{n}\Vert
x\Vert \\
& \leq\left( c_{\alpha}+\dfrac{1}{e^{\alpha}-1}\right) M_{n}\Vert x\Vert,
\end{align*}
which proves (\ref{eq4}).

\textbf{Step 2}. We prove that
\begin{equation}
\Vert A(m,n)\Vert\leq\lbrack c_{\alpha}(e^{\alpha}-1)+1]\,\dfrac{c_{\alpha
}^{k}\,k!}{(m-n)^{k}}\,M_{n},\text{ for }m>n\text{ in }\mathbb{N},
\label{eq5}
\end{equation}
and for all $k\in\mathbb{N}$.

Step 1 shows that \eqref{eq5} is valid for $k=0$. Assume now that \eqref{eq5}
holds for some $k\in\mathbb{N}$. For fixed non-negative integers $m>n$ and $%
x\in E$ we define a sequence $v=(v_{i})_{i\in\mathbb{N}}$ by
\begin{equation*}
v_{i}=%
\begin{cases}
(i-n)^{k}A(i,n)x, & i\in\{n+1,\ldots,m\}, \\
0, & \text{otherwise}.%
\end{cases}
\end{equation*}
Notice that $v\in c_{00}(\alpha)$. For $i\in\{n+1,\ldots,m\}$ one gets
\begin{equation*}
\Vert
v_{i}\Vert_{i,\alpha}=\sup\limits_{j\in\Delta_{i}}e^{-\alpha(j-i)}\Vert
A(j,i)v_{i}\Vert=\sup\limits_{j\in\Delta_{i}}e^{-\alpha(j-i)}(i-n)^{k}\Vert
A(j,n)x\Vert.
\end{equation*}
Using the induction step we have
\begin{equation*}
\Vert v_{i}\Vert_{i,\alpha}\leq\sup\limits_{j\in\Delta_{n}}(j-n)^{k}\Vert
A(j,n)x\Vert\leq\lbrack
c_{\alpha}(e^{\alpha}-1)+1]\,c_{\alpha}^{k}\,k!\,M_{n}\Vert x\Vert.
\end{equation*}
If $i\in\{0,\ldots,n\}$ or $i\geq m+1$, then $\Vert v_{i}\Vert_{i,\alpha}=0$%
, therefore
\begin{equation*}
\label{eq.t1-1}\Vert v\Vert_{\alpha}\leq\lbrack c_{\alpha}(e^{\alpha
}-1)+1]\,c_{\alpha}^{k}\,k!\,M_{n}\Vert x\Vert.
\end{equation*}
Put $u=G_{\alpha}^{-1}(-v)\in c_{00}(\alpha)$. Lemma \ref{lem2-1} (i)
implies
\begin{equation*}
u_{m}=\sum\limits_{i=0}^{m}A(m,i)v_{i}=\left(
\sum\limits_{i=n+1}^{m}(i-n)^{k}\right) A(m,n)x=\left(
\sum\limits_{i=1}^{m-n}i^{k}\right) A(m,n)x,
\end{equation*}
and from inequality \eqref{eq.c} in Appendix \ref{ap.2} we obtain
\begin{equation*}
\label{eq.t1-2}\Vert u_{m}\Vert\geq\dfrac{(m-n)^{k+1}}{k+1}\,\Vert
A(m,n)x\Vert.
\end{equation*}
Estimating%
\begin{equation*}
\dfrac{(m-n)^{k+1}}{k+1}\,\Vert A(m,n)x\Vert\leq\Vert u\Vert_{\alpha}\leq
c_{\alpha}\Vert v\Vert_{\alpha}\leq\lbrack
c_{\alpha}(e^{\alpha}-1)+1]\,c_{\alpha}^{k+1}\,k!\,M_{n}\Vert x\Vert,
\end{equation*}
we conclude that \eqref{eq5} holds for $k+1$.

\textbf{Step 3}. We show that%
\begin{equation}
\Vert A(m,n)\Vert\leq2[c_{\alpha}(e^{\alpha}-1)+1]M_{n}e^{-\frac{1}{%
2c_{\alpha}}(m-n)},\text{ for }(m,n)\in\Delta.   \label{eq6}
\end{equation}

If $(m,n)\in\Delta$, from Step 2 one gets
\begin{equation*}
\left( \dfrac{m-n}{c_{\alpha}}\right) ^{k}\dfrac{1}{k!}\,\Vert
A(m,n)\Vert\leq\lbrack c_{\alpha}(e^{\alpha}-1)+1]\,M_{n},\;k\in\mathbb{N}.
\end{equation*}
Multiplying the above inequality by $\dfrac{1}{2^{k}}$ and summing over $k$
we obtain estimation \eqref{eq6}, thus $\mathcal{A}$ is nonuniform
exponentially stable.
\end{proof}

The sufficiency of Theorem \ref{th1} furnishes a simple condition for the
existence of nonuniform exponential stability in the case of discrete
dynamics. Recall that the \emph{spectral radius} of a bounded linear
operator $T:E\rightarrow E$ is defined by
\begin{equation*}
r\left( T\right) =\lim\limits_{n\rightarrow\infty}\sqrt[n]{\left\Vert
T^{n}\right\Vert }.
\end{equation*}

\begin{corollary}
\label{cor.suf}If for some $\alpha\in\mathcal{E}_{ad}(\mathcal{A})$ the
spectral radius $r(T_{\alpha})<1$, then the sequence $\mathcal{A}$ is
nonuniform exponentially stable.
\end{corollary}

Our second main result below offers a necessary and sufficient condition for
the invertibility of each operator $G_{\alpha}$, $\alpha\in\mathcal{E}_{ad}(%
\mathcal{A})$. At our best knowledge this result is quite new, being
specific to nonuniform asymptotic behavior.

\begin{theorem}
\label{t1}Let $\alpha\in\mathcal{E}_{ad}(\mathcal{A})$. The operator $%
G_{\alpha}$ is invertible on $c_{00}(\alpha)$ if and only if there exists $%
\beta<0$, $\beta\in\mathcal{E}_{ad}(\mathcal{A})$, such that $%
c_{00}(\alpha)=c_{00}(\beta)$.
\end{theorem}

\begin{proof}
Sufficiency. Since $c_{00}(\alpha)=c_{00}(\beta)$ then from one hand $%
G_{\alpha}=G_{\beta}$ and from another hand, as $\beta<0$, from the proof of
the necessity of the previous theorem we get that $G_{\beta}$ is invertible.
Thus $G_{\alpha}$ is invertible.

Necessity. According to the above arguments it suffices to analyze only the
case $\alpha\geq0$. Assume that $G_{\alpha}$ is invertible on $c_{00}(\alpha
)$. The sufficiency of Theorem \ref{th1} implies that the sequence $\mathcal{%
A}$ is nonuniform exponentially stable, that is $\mathcal{E}_{ad}(\mathcal{A}%
)$ contains negative admissible exponents. We start our proof with some
auxiliary results.

\textbf{Step 1}. Basic constructions.

For chosen $\nu\in\mathcal{E}_{ad}(\mathcal{A})$ with $\nu<0$, pick $%
\varepsilon>0$ such that $\beta:=\nu+\varepsilon\in(-1,0)$. It follows that $%
\beta\in\mathcal{E}_{ad}(\mathcal{A})$. For any fixed $n\in\mathbb{N}$ and $%
x\in E\setminus\{0\}$ we define the sequence $\theta=\theta(\beta
,n,x)=(\theta_{k})_{k\in\mathbb{N}}$ in $E$\ by%
\begin{equation*}
\theta_{k}=%
\begin{cases}
e^{-\beta(k-n)}A(k,n)x, & k\geq n+1, \\
0, & k\in\{0,\ldots,n\}.%
\end{cases}
\end{equation*}
Let us notice that $\theta\in c_{00}(\alpha)$. Indeed, for $k\geq n+1$ we
have%
\begin{align*}
\| \theta_{k}\|_{k,\alpha} & =\sup\limits_{m\in\Delta_{k}}e^{-\alpha
(m-k)}\Vert A(m,k)\theta_{k}\Vert=\sup\limits_{m\in\Delta_{k}}e^{-\alpha
(m-k)}e^{-\beta(k-n)}\Vert A(m,n)x\Vert \\
& \leq\sup\limits_{m\in\Delta_{k}}e^{-\beta(k-n)}\Vert A(m,n)x\Vert\leq
\sup\limits_{m\in\Delta_{k}}e^{-\beta(k-n)}e^{\nu(m-n)}\,M_{n}\,\Vert x\Vert
\\
& \leq e^{-\beta(k-n)}e^{\nu(k-n)}\,M_{n}\,\Vert x\Vert=e^{-\varepsilon
(k-n)}\,M_{n}\,\Vert x\Vert,
\end{align*}
which proves that $\theta\in c_{00}(\alpha)$. Since $\lim\limits_{k%
\rightarrow \infty}\| \theta_{k}\|_{k,\alpha}=0$, there exists (not
necessarily unique) $k_{0}\in\mathbb{N}$ satisfying
\begin{equation*}
\sup\limits_{k\in\mathbb{N}}\| \theta_{k}\|_{k,\alpha}=\|
\theta_{k_{0}}\|_{k_{0},\alpha},\label{eq.Tp-1}
\end{equation*}
and let $p=p(\beta,n,x)\in\mathbb{N}$ be the minimum of all such $k_{0}$.
Remark that $p\geq n+1$.

\textbf{Step 2}. There exists $\beta\in\mathcal{E}_{ad}(\mathcal{A}%
)\cap(-1,0)$ such that for any $n\in\mathbb{N}$ and $x\in E\setminus\{0\}$,
one has $p=p(\beta,n,x)\in\left[ n+1,n-\frac{1}{\beta}\right) $.

Assuming the contrary we deduce that for any fixed integer $k\geq2$, large
enough such that $\beta_{k}=-\frac{1}{k}\in\mathcal{E}_{ad}(\mathcal{A})$,
there exist $n_{k}\in\mathbb{N}$ and $x_{k}\in E\setminus\left\{ 0\right\} $
for which $p_{k}=p\left( -\frac{1}{k},n_{k},x_{k}\right) \geq n_{k}+k$. Let
us denote $y=y(k)=\theta\left( -\frac{1}{k},n_{k},x_{k}\right) \in
c_{00}(\alpha)$ and let $v=v(k)=\frac{1}{k}\,y\in c_{00}(\alpha)$. We intend
to prove that the sequence $u=(u_{m})_{m\in\mathbb{N}}$ defined by
\begin{equation*}
u_{m}=%
\begin{cases}
\dfrac{1}{k}\sum\limits_{j=1}^{m-n_{k}}e^{\frac{j}{k}}\,A(m,n_{k})x_{k}, &
m\geq n_{k}+1, \\
0, & m\in\{0,\ldots,n_{k}\},%
\end{cases}
\end{equation*}
belongs to $c_{00}(\alpha)$. From inequality
\begin{equation*}
\dfrac{e^{t}-1}{t}<e^{t}<\dfrac{e^{t}(e^{t}-1)}{t},\;t>0,
\end{equation*}
substituting $t=1/k$ one gets
\begin{equation*}
1<\dfrac{e^{\frac{1}{k}}}{k(e^{\frac{1}{k}}-1)}<e^{\frac{1}{k}}.
\end{equation*}
For $m\geq n_{k}+1$, multiplying the above relation by $e^{\frac{m-n_{k}}{k}%
}-1$, we have
\begin{equation}
e^{\frac{m-n_{k}}{k}}-1<\dfrac{1}{k}\sum\limits_{j=1}^{m-n_{k}}e^{\frac{j}{k}%
}<e^{\frac{m-n_{k}+1}{k}}-e^{\frac{1}{k}}.   \label{eq7}
\end{equation}
The right hand inequality (\ref{eq7}) readily implies
\begin{align*}
\Vert u_{m}\Vert_{m,\alpha} & =\left( \dfrac{1}{k}\sum%
\limits_{j=1}^{m-n_{k}}e^{\frac{j}{k}}\right)
\,\sup\limits_{i\in\Delta_{m}}e^{-\alpha(i-m)}\Vert A(i,n_{k})x_{k}\Vert \\
& <e^{\frac{1}{k}}\left( 1-e^{-\frac{m-n_{k}}{k}}\right) \,\sup
\limits_{i\in\Delta_{m}}e^{-\alpha(i-m)}e^{\frac{m-n_{k}}{k}}\Vert
A(i,n_{k})x_{k}\Vert \\
& =e^{\frac{1}{k}}\left( 1-e^{-\frac{m-n_{k}}{k}}\right) \Vert
y_{m}\Vert_{m,\alpha}.
\end{align*}
Since $y\in c_{00}(\alpha)$, it follows that $u\in c_{00}(\alpha)$, moreover
\begin{equation*}
\Vert u\Vert_{\alpha}\leq e^{\frac{1}{k}}\Vert y\Vert_{\alpha}.
\end{equation*}
On the other hand, one can easily see that $u_{m}=\sum%
\limits_{j=0}^{m}A(m,j)v_{j}$, $m\in\mathbb{N}$. Lemma \ref{lem2-1} (i)
yields $G_{\alpha }u=-v$. Using the left hand inequality in \eqref{eq7},
under hypothesis $p_{k}\geq n_{k}+k$ we successively have
\begin{align*}
\Vert u\Vert_{\alpha} & \geq\Vert u_{p_{k}}\Vert_{p_{k},\alpha}>\left( 1-e^{-%
\frac{p_{k}-n_{k}}{k}}\right) \Vert y_{p_{k}}\Vert_{p_{k},\alpha} \\
& =\left( 1-e^{-\frac{p_{k}-n_{k}}{k}}\right) \Vert y\Vert_{\alpha} \\
& \geq\left( 1-e^{-1}\right) \Vert y\Vert_{\alpha}.
\end{align*}
As%
\begin{equation*}
\frac{\Vert u\Vert_{\alpha}}{\Vert v\Vert_{\alpha}}=\frac{\Vert y\Vert
_{\alpha}}{\Vert v\Vert_{\alpha}}\frac{\Vert u\Vert_{\alpha}}{\Vert
y\Vert_{\alpha}}=k\frac{\Vert u\Vert_{\alpha}}{\Vert y\Vert_{\alpha}}%
>k\left( 1-e^{-1}\right) ,
\end{equation*}
it turns out that
\begin{equation*}
\frac{\Vert G_{\alpha}^{-1}v\Vert_{\alpha}}{\Vert v\Vert_{\alpha}}=\frac{%
\Vert u\Vert_{\alpha}}{\Vert v\Vert_{\alpha}}>k\left( 1-e^{-1}\right) ,
\end{equation*}
which contradicts the boundedness of the operator $G_{\alpha}^{-1}$ on $%
c_{00}(\alpha)$. It follows that the statement in Step 2 holds.

\textbf{Step 3}. The proof of the theorem.

Let $\beta\in\mathcal{E}_{ad}(\mathcal{A})$ be the admissible exponent in
Step 2. Since $\beta<\alpha$, to get $c_{00}(\alpha)=c_{00}(\beta)$, by Remark \ref{r1} (ii), it
suffices to prove the inclusion $c_{00}(\alpha)\subset c_{00}(\beta)$.
Indeed, let $u\in c_{00}(\alpha)$ and pick $n\in\mathbb{N}$. If $u_{n}\neq0$
then for all $m\in\Delta_{n}$ one has
\begin{equation*}
e^{-\beta(m-n)}\Vert A(m,n)u_{n}\Vert=\Vert\theta_{m}\Vert\leq\Vert\theta
_{m}\Vert_{m,\alpha}\leq\Vert\theta\Vert_{\alpha}=\Vert\theta_{p}\Vert_{p,%
\alpha},
\end{equation*}
with $\theta=\theta(\beta,n,u_{n})$ and $p=p(\beta,n,u_{n})$ being defined
in Step 1. We estimate
\begin{align*}
\sup\limits_{m\in\Delta_{n}}e^{-\beta(m-n)}\Vert A(m,n)u_{n}\Vert & \leq
\Vert\theta_{p}\Vert_{p,\alpha}=\sup\limits_{m\in\Delta_{p}}e^{(\alpha
-\beta)(p-n)}e^{-\alpha(m-n)}\Vert A(m,n)u_{n}\Vert \\
& \leq e^{-(\alpha-\beta)\frac{1}{\beta}}\sup\limits_{m\in\Delta_{n}}e^{-%
\alpha(m-n)}\Vert A(m,n)u_{n}\Vert.
\end{align*}
From the above inequality, observing that it still holds for $u_{n}=0$, we
deduce that $\Vert u_{n}\Vert_{n,\beta}\leq e^{-(\alpha-\beta)\frac{1}{\beta}%
}\Vert u_{n}\Vert_{n,\alpha}$, and this implies $c_{00}(\alpha)\subset
c_{00}(\beta)$.
\end{proof}

\section{Nonuniform exponential stability of discrete bounded orbits}

\label{sec.4}

From Section \ref{sec.3} we already know that a spectral property of the
operator\ $G_{\alpha}$ can completely characterize nonuniform exponential
stability. In this section we expose a sufficient condition for the
existence of nonuniform exponential stability of all bounded orbits of $%
\mathcal{A}$, by weakening the condition $0\notin\sigma(G_{\alpha})$. Recall
that the \emph{approximate point spectrum} of a bounded linear operator $%
T:E\rightarrow E$ is the set
\begin{equation*}
\sigma_{ap}(T)=\{\lambda\in\mathbb{C}:\,\forall\,\varepsilon>0\;\exists\,x%
\in E\text{ with }\Vert x\Vert=1\text{ such that }\Vert\lambda x-Tx\Vert
\leq\varepsilon\}.
\end{equation*}

\begin{theorem}
\label{t2}Assume that $0\notin\sigma_{ap}(G_{\alpha})$, for some $\alpha \in%
\mathcal{E}_{ad}(\mathcal{A})$. If
\begin{equation*}
\sup\limits_{m\in\Delta_{n}}\Vert A(m,n)x_{0}\Vert<\infty,
\end{equation*}
for some fixed $n\in\mathbb{N}$ and $x_{0}\in E$, then there exist a
constant $\nu>0$ and a positive real sequence $(N_{k})_{k\in\mathbb{N}}$,
independent of $n$ and $x_{0}$, such that
\begin{equation*}
\Vert A(m,n)x_{0}\Vert\leq N_{n}e^{-\nu(m-n)}\Vert x_{0}\Vert,\text{ for }%
m\in\Delta_{n}.
\end{equation*}
\end{theorem}

\begin{proof}
It suffices to consider the case $\alpha\geq0$. Since $0\notin\sigma
_{ap}(G_{\alpha})$, there exists $\theta>0$ (depending only on $\alpha$)
such that
\begin{equation*}
\Vert u\Vert_{\alpha}\leq\theta\,\Vert G_{\alpha}u\Vert_{\alpha},\text{ for
all }u\in c_{00}(\alpha).
\end{equation*}
Pick $n\in\mathbb{N}$ and $x_{0}\in E$ for which $K(n,x_{0}):=\sup
\limits_{m\in\Delta_{n}}\Vert A(m,n)x_{0}\Vert<\infty$. We divide the proof
into several steps.

\textbf{Step1}. The following estimation holds:
\begin{equation}
\Vert A(m,n)x_{0}\Vert\leq\frac{\theta}{m-n}\,K(n,x_{0}),\text{ for every
integer }m>n.   \label{eq9}
\end{equation}

For any fixed integer $m>n$ we define the real sequence $a=(a_{k})_{k\in%
\mathbb{N}}$ by
\begin{equation*}
a_{k}=%
\begin{cases}
0, & k\in\{0,\ldots,n\}, \\
\frac{k-n}{(m-n)^{2}}, & k\in\{n+1,\ldots,m\}, \\
\frac{1}{k-n}, & k\geq m+1.%
\end{cases}
\end{equation*}
Using the above sequence $a=(a_{k})_{k\in\mathbb{N}}$ we construct the
(vector valued) sequences $u=u(m,n,x_{0})=(u_{k})_{k\in\mathbb{N}}$ and $%
v=v(m,n,x_{0})=(v_{k})_{k\in\mathbb{N}}$ defined by
\begin{equation*}
u_{k}=%
\begin{cases}
a_{k}A(k,n)x_{0}, & k\geq n+1, \\
0, & k\in\{0,\ldots,n\},%
\end{cases}
\end{equation*}
and
\begin{equation*}
v_{k}=%
\begin{cases}
(a_{k}-a_{k-1})A(k,n)x_{0}, & k\geq n+1, \\
0, & k\in\{0,\ldots,n\}.%
\end{cases}
\end{equation*}
One can easily prove that $u,v\in c_{00}(\alpha)$. Moreover, if $k\in
\{0,\ldots,n\}$ we see that $\sum\limits_{i=0}^{k}A(k,i)v_{i}=0=u_{k}$,
meanwhile for $k\geq n+1$ we have%
\begin{equation*}
\sum\limits_{i=0}^{k}A(k,i)v_{i}=\sum%
\limits_{i=n+1}^{k}(a_{i}-a_{i-1})A(k,n)x_{0}=a_{k}A(k,n)x_{0}=u_{k}.
\end{equation*}
By Lemma \ref{lem2-1} (i) it follows that $G_{\alpha}u=-v$. For $k\geq n+1$
we have%
\begin{equation*}
\|v_{k}\|_{k,\alpha}=|a_{k}-a_{k-1}|\sup\limits_{i\in\Delta_{k}}e^{-%
\alpha(i-k)}\Vert A(i,n)x_{0}\Vert\leq|a_{k}-a_{k-1}|K(n,x_{0}).
\end{equation*}
A simple computation shows that $|a_{k}-a_{k-1}|\leq\frac{1}{(m-n)^{2}}$,
and thus
\begin{equation*}
\Vert v\Vert_{\alpha}\leq\frac{1}{(m-n)^{2}}\,K(n,x_{0}).
\end{equation*}
Now, we successively have
\begin{equation*}
\frac{1}{m-n}\Vert A(m,n)x_{0}\Vert=\Vert u_{m}\Vert\leq\Vert u\Vert_{\alpha
}\leq\theta\,\Vert G_{\alpha}u\Vert_{\alpha}=\theta\,\Vert
v\Vert_{\alpha}\leq\frac{\theta}{(m-n)^{2}}\,K(n,x_{0}).
\end{equation*}
The above estimation directly proves \eqref{eq9}.

\textbf{Step 2}. If $(M_{i})_{i\in\mathbb{N}}$ is the sequence given by %
\eqref{eq1}, then there exists a constant $L>0$ (depending only on $\alpha$)
such that
\begin{equation}  \label{eq10}
\| A(m,n) x_{0} \| \leq\frac{\theta L}{m-n}M_{n}\,\| x_{0}\|, \text{ for
every integer } m>n.
\end{equation}

Estimation \eqref{eq9} yields $\lim\limits_{m\to\infty} \| A(m,n) x_{0}\| =0$.
Thus, there exists a non-negative integer $m_{0}=m_{0}(n,x_{0})\geq n$ for
which $$K(n,x_0)=\sup\limits_{ m\in\Delta_{n}} \| A(m,n)x_{0}\| =\|A(m_{0},n)x_{0}\| .$$
The same inequality \eqref{eq9} also implies $(m_{0}-n)\| A(m_{0},n) x_{0}
\| \leq\theta\, \| A(m_{0},n) x_{0}\| $, which is equivalent to $%
m_{0}-n\leq\theta$. For each positive integer $m>n$, one has
\begin{align*}
\| A(m,n)x_{0}\| & \leq\frac{\theta}{m-n}\| A(m_{0},n)x_{0}\| \\
& \leq\frac{\theta}{m-n}\,e^{\alpha(m_{0}-n)} M_{n} \,\|x_{0}\| \\
& \leq\frac{\theta e^{\alpha\theta}}{m-n}\, M_{n} \,\|x_{0}\|.
\end{align*}
This shows that \eqref{eq10} holds for $L=e^{\alpha\theta}$.

\textbf{Step 3}. For each $k\in\mathbb{N}$, the following
inequality holds:
\begin{equation}  \label{eq11}
\| A(m,n)x_{0}\|\leq\frac{k!\theta^{k}}{(m-n)^{k}}\,\widetilde {N}_{n} \|
x_{0} \|, \text{ for every integer } m>n,
\end{equation}
where $\widetilde{N}_{n}=\theta L M_{n}$, with $L>0$ given by Step 2.

From Step 2 it readily follows that
$$ \| A(m,n)x_{0}\|\leq \frac{\widetilde{N}_n}{m-n} \|x_0\|\leq \widetilde{N}_n \|x_0\|,$$
and thus \eqref{eq11} holds for $k=0$. Fix now $k\in\mathbb{N}^{*}$
and assume that \eqref{eq11} is valid for $k-1$. For any fixed integer $m>n$
we define a sequence $b=b(m,n,k)=(b_{i})_{i\in\mathbb{N}}$ by
\begin{equation*}
b_{i}=
\begin{cases}
0, & i\in\{0,\ldots,n\}, \\
(i-n)^{k}, & i\in\{ n+1,\dots,m\}, \\
\frac{(m-n)^{k+2}}{(i-n)^{2}}, & i\geq m+1.%
\end{cases}
\end{equation*}
Using $b$ we construct sequences $u=(u_{i})_{i\in\mathbb{N}}$ and $%
v=(v_{i})_{i\in\mathbb{N}}$ given by
\begin{equation*}
u_{i}=
\begin{cases}
b_{i} A(i,n)x_{0}, & i\geq n+1, \\
0, & i\in\{0,\ldots,n\},%
\end{cases}
\end{equation*}
and
\begin{equation*}
v_{i}=
\begin{cases}
(b_{i}-b_{i-1}) A(i,n)x_{0}, & i\geq n+1, \\
0, & i\in\{0,\ldots,n\}.%
\end{cases}
\end{equation*}
As in Step 1, one may show that $G_{\alpha}u=-v$. If $i\in\{0,\ldots,n\}$,
then $\|v_{i}\|_{i,\alpha} =0$. For $i\geq n+1$, let us estimate
\begin{align*}
\|v_{i}\|_{i,\alpha} & \leq|b_{i}-b_{i-1}| \sup\limits_{j\in\Delta_{i}} \|
A(j,n) x_{0}\| \\
& \leq k (i-n)^{k-1} \sup\limits_{j\in\Delta_{i}} \frac{(k-1)!\,\theta^{k-1}%
}{(j-n)^{k-1}}\,\widetilde{N}_{n} \| x_{0}\| \\
& = k! \, \theta^{k-1} \widetilde{N}_{n} \| x_{0}\|.
\end{align*}
It follows that $\|v\|_{\alpha}\leq k! \, \theta^{k-1} \widetilde{N}_{n} \|
x_{0}\|$, and thus
\begin{equation*}
(m-n)^{k} \| A(m,n)x_{0}\| = \| u_{m} \| \leq\| u\|_{\alpha}\leq\theta\|
v\|_{\alpha}\leq k! \, \theta^{k} \widetilde{N}_{n} \| x_{0}\|.
\end{equation*}
This last estimation readily implies that \eqref{eq11} holds for $k$, which
proves the claim.

\textbf{Step 4}. For any fixed $\delta\in(0,1)$ one has
\begin{equation}  \label{eq.fin}
\| A(m,n)x_{0}\|\leq\frac{1}{1-\delta} \, \widetilde{N}_{n} e^{-\frac{\delta%
}{\theta}(m-n)}\| x_{0}\|, \text{ for every } m\in\Delta_{n}.
\end{equation}

Pick $\delta\in(0,1)$. Multiplying inequality \eqref{eq11} by $\delta^{k}$, $%
k\in\mathbb{N}$, we get
\begin{equation*}
\frac{1}{k!}\left[ \frac{\delta}{\theta}(m-n) \right] ^{k} \| A(m,n)x_{0}\|
\leq\delta^{k} \widetilde{N}_{n} \| x_{0}\|, \; m\in\Delta_{n}, \, k\in%
\mathbb{N},
\end{equation*}
and summing with respect to $k\in\mathbb{N}$, we obtain
\begin{equation*}
e^{\frac{\delta}{\theta} (m-n)} \| A(m,n)x_{0}\|\leq\frac{1}{1-\delta }\,%
\widetilde{N}_{n} \| x_{0}\|, \text{ for } m\in\Delta_{n},
\end{equation*}
which is equivalent to \eqref{eq.fin}, and this completes the proof.
\end{proof}

\section{Applications and examples}

\subsection{Applications to uniform exponential stability}

\label{sec.5}

As consequences of our conclusions, in what follows we deduce discrete
versions of some known results, and we especially refer to papers \cite{Bu},
\cite{Hu.Mi.2001}, and \cite{Mi.Ra.Sc.1998}.

The below result emphasizes the connection between uniform and nonuniform
behavior.

\begin{proposition}
\label{P1} Let $\alpha\in\mathcal{E}_{ad}(\mathcal{A})$. $\mathcal{A}$ is $%
\alpha$-uniform exponentially bounded if and only if $c_{00}(\alpha)=c_{00}$.
\end{proposition}

\begin{proof}
\emph{Sufficiency}. Assume that $c_{00}(\alpha)=c_{00}$ for some $\alpha \in%
\mathcal{E}_{ad}(\mathcal{A})$. If $\mathcal{A}$ is not $\alpha$-uniform
exponentially bounded, then for each $n\in\mathbb{N}$ there exist $u_{n}\in E
$, $u_{n}\neq0$, and integers $k_{n}\geq l_{n}\geq0$ that satisfy estimation
\begin{equation}
\Vert A(k_{n},l_{n})u_{n}\Vert>ne^{\alpha(k_{n}-l_{n})}\Vert u_{n}\Vert.
\label{eq.est}
\end{equation}
If the sequence $(l_{n})_{n\in\mathbb{N}}$ is bounded, then $l_{n}\leq l$
for some $l\in\mathbb{N}$, thus
\begin{equation*}
n\Vert u_{n}\Vert<e^{-\alpha(k_{n}-l_{n})}\Vert A(k_{n},l_{n})u_{n}\Vert\leq
M_{l_{n}}\Vert u_{n}\Vert\leq\max\limits_{k\in\{0,\ldots,l\}}M_{k}\Vert
u_{n}\Vert,
\end{equation*}
which is a contradiction. If $(l_{n})_{n\in\mathbb{N}}$ is strictly
increasing, let $(\alpha_{n})_{n\in\mathbb{N}}$ be a real sequence such that
$\alpha_{n}\rightarrow0$ and $n|\alpha_{n}|\rightarrow\infty$ as $%
n\rightarrow\infty$. Consider a sequence $v=(v_{n})_{n\in\mathbb{N}}\in
c_{00}$ for which its subsequence $(v_{l_{n}})_{n\in\mathbb{N}}$ is given by
\begin{equation*}
v_{l_{n}}=\alpha_{n}\,\dfrac{u_{n}}{\Vert u_{n}\Vert},\text{ for }n\in%
\mathbb{N}^{\ast}.
\end{equation*}
The hypothesis $c_{00}(\alpha)=c_{00}$ implies $v\in c_{00}(\alpha)$. On the
other hand, according to \eqref{eq.est} we have
\begin{equation*}
n|\alpha_{n}|< e^{-\alpha(k_{n}-l_{n})}\Vert
A(k_{n},l_{n})v_{l_{n}}\Vert\leq\Vert v_{l_{n}}\Vert_{l_{n},\alpha}.
\end{equation*}
The above estimation yields $v\notin c_{00}(\alpha)$, that contradicts our
assumption. We thereby conclude that $\mathcal{A}$ is $\alpha$-uniform
exponentially bounded.

\emph{Necessity}. Assume that there exists a constant $M>0$ such that
\begin{equation*}
\Vert A(m,n)\Vert\leq Me^{\alpha(m-n)},\text{ for all }(m,n)\in\Delta.
\end{equation*}
From Remark \ref{r1} (i), it suffices to prove the inclusion $c_{00}\subset
c_{00}(\alpha)$. In this regard, let us notice that for every $u\in c_{00}$
one gets $\Vert u_{n}\Vert_{n,\alpha}\leq M\Vert u_{n}\Vert\rightarrow0$ as $%
n\rightarrow0$, therefore $u\in c_{00}(\alpha)$.
\end{proof}

Using Proposition \ref{P1} and Theorem \ref{th1}, one can obtain a necessary
and sufficient condition for the existence of uniform exponential stability,
in terms of invertibility of the operator $G$ restricted on $c_{00}$.

\begin{corollary}
\label{cor.u1} If the sequence $\left( A_{n}\right) _{n\in\mathbb{N}}$ is
uniform exponentially bounded, then it is uniform exponentially stable if
and only if the operator $G$ is invertible on $c_{00}$.
\end{corollary}

The above result can be considered a discrete version of Theorem 2.2 in \cite%
{Mi.Ra.Sc.1998} (see also Theorem 3.4 in \cite{Bu}).

Reader will certainly notice that the statement in Theorem 2.1 from \cite%
{Hu.Mi.2001} exposed below, is also an immediate consequence of our Theorem %
\ref{t2}.

\begin{corollary}
\label{cor.u2} Let $\mathcal{A}$ be a uniform exponentially bounded sequence
and assume that $0\notin\sigma_{ap}(G)$. If $\sup\limits_{m\in\Delta_{n}} \|
A(m,n)x_{0} \| <\infty$ for some fixed $n\in\mathbb{N}$ and $x_{0}\in E$,
then there exist constants $N,\nu>0$, independent of $n$ and $x_{0}$, such
that
\begin{equation*}
\| A(m,n)x_{0}\| \leq N e^{-\nu(m-n)}\|x_{0}\|, \text{ for all } m\in
\Delta_{n}.
\end{equation*}
\end{corollary}

\subsection{Examples}

Let us emphasize an important difference between uniform and nonuniform
behavior: while in the uniform case $c_{00}(\alpha)=c_{00}$ for each
admissible exponent $\alpha$ (see Proposition \ref{P1}), in the nonuniform
setting there are examples of discrete dynamics for which $c_{00}(\alpha)$
is not unique. Moreover, there exist discrete families admitting infinitely
many such sequence spaces, as illustrated in Example \ref{ex.2}.

In the following particular situations we identify the set $\mathcal{E}_{ad}(%
\mathcal{A})$ for some operator sequences $\mathcal{A}$, pointing out the
connections between the spaces $c_{00}(\alpha)$, when $\alpha$ varies on $%
\mathcal{E}_{ad}(\mathcal{A})$.

\begin{example}
\label{ex1} Let $A_{n}=e^{a_{n}-a_{n+1}}\mathrm{Id}$, where $a_{n}=\dfrac {n%
}{2+(-1)^{n}}$, $n\in\mathbb{N}$. The discrete evolution family associated
to $\mathcal{A}=(A_{n})_{n\in\mathbb{N}}$ is given by
\begin{equation*}
A(m,n)=e^{a_{n}-a_{m}}\mathrm{Id},\text{ for }(m,n)\in\Delta.
\end{equation*}
The discrete dynamics $\mathcal{A}$ is at the same time not uniformly
exponentially bounded and nonuniform exponentially stable, with $\mathcal{E}%
_{ad}(\mathcal{A})=[-1/3,\infty)$. Moreover,
\begin{equation}
c_{00}(\alpha)=c_{00}(-1/3),\text{ }\alpha\in\mathcal{E}_{ad}(\mathcal{A}).
\label{eq.ex1}
\end{equation}
\end{example}

\begin{proof}
For all $(m,n)\in\Delta$ one has
\begin{align*}
a_{n}-a_{m} & =m\left( \frac{1}{3}-\frac{1}{2+(-1)^{m}}\right) -n\left(
\frac{1}{3}-\frac{1}{2+(-1)^{n}}\right) -\frac{1}{3}\left( m-n\right) \\
& =-\frac{m}{3}\cdot\frac{1-(-1)^{m}}{2+(-1)^{m}}+\frac{n}{3}\cdot \frac{%
1-(-1)^{n}}{2+(-1)^{n}}-\frac{1}{3}\left( m-n\right) .
\end{align*}

Since
\begin{equation*}
\left\Vert A_{2n-1}\right\Vert =e^{a_{2n-1}-a_{2n}}=e^{\left( 2n-1\right)
\frac{2}{3}-\frac{1}{3}}\rightarrow\infty,
\end{equation*}
it follows that $\mathcal{A}$ is not uniform exponentially bounded. Notice
that inequality $e^{\frac{1}{3}\left( m-n\right) }\Vert A(m,n)\Vert\leq e^{%
\frac{2}{3}n}$ for $(m,n)\in\Delta$ implies $-\frac{1}{3}\in\mathcal{E}_{ad}(%
\mathcal{A})$. Fix $\varepsilon>0$. From estimation%
\begin{equation*}
e^{(\frac{1}{3}+\varepsilon)\left( 2n-0\right) }\Vert A(2n,0)\Vert =e^{-%
\frac{2n}{3}+(\frac{1}{3}+\varepsilon)2n}=e^{2n\varepsilon}\rightarrow%
\infty,
\end{equation*}
we deduce that $-\frac{1}{3}-\varepsilon\notin\mathcal{E}_{ad}(\mathcal{A})$%
, in conclusion $\mathcal{E}_{ad}(\mathcal{A})=[-1/3,\infty)$.

To prove \eqref{eq.ex1} it suffices to show that $c_{00}(\alpha)\subset
c_{00}(-1/3)$, for all $\alpha\geq-\frac{1}{3}$. For such fixed $\alpha$ and
for any $(m,n)\in\Delta$ we put
\begin{equation*}
f_{\alpha}(m,n)=a_{n}-a_{m}-\alpha(m-n),
\end{equation*}
and rewriting
\begin{equation*}
f_{\alpha}(m,n)=-\frac{m}{3}\cdot\frac{1-(-1)^{m}}{2+(-1)^{m}}+\frac{n}{3}%
\cdot\frac{1-(-1)^{n}}{2+(-1)^{n}}-\left( \alpha+\frac{1}{3}\right) \left(
m-n\right) ,
\end{equation*}
notice that
\begin{equation}\label{eq.expl}
f_{n}=\sup\limits_{m\in\Delta_{n}}f_{-1/3}(m,n)=
\begin{cases}
0, &\text{if } n \text{ is  even},\\
\frac{2n}{3}, &\text{if } n \text{ is  odd}.
\end{cases}
\end{equation}
For each $n\in\mathbb{N}$ let us denote $m_{n}$ the (unique) even integer in
the set $\{n,n+1\}$. It is easy to check that $f_{\alpha}(m_{n},n)\in\left\{
0,\frac{2n}{3}-\left( \alpha+\frac{1}{3}\right) \right\} $.\ Then for each $%
u\in c_{00}(\alpha)$ we successively have
\begin{align*}
\Vert u_{n}\Vert_{n,-1/3} & =\sup\limits_{m\in\Delta_{n}}e^{\frac{1}{3}(m-n)}\Vert
A(m,n)u_{n}\Vert \\
& =\sup\limits_{m\in\Delta_{n}}e^{f_{-1/3}(m,n)}\Vert u_{n}\Vert\\
& =e^{f_{n}}\Vert u_{n}\Vert\\
&=e^{f_{n}-\sup\limits_{m\in\Delta_{n}}f_{\alpha
}(m,n)}\Vert u_{n}\Vert_{n,\alpha} \\
& \leq e^{f_{n}-f_{\alpha}(m_{n},n)}\Vert u_{n}\Vert_{n,\alpha}. \\
\end{align*}
By \eqref{eq.expl} we deduce that
\[
f_{n}-f_{\alpha}(m_{n},n)=
\begin{cases}
0, &\text{if } n \text{ is  even},\\
\alpha+1/3, &\text{if } n \text{ is  odd},
\end{cases}
\]
and thus $\Vert u_{n}\Vert_{n,-1/3}\leq e^{\alpha+\frac{1}{3}}\Vert u_{n}\Vert_{n,\alpha},$
which readily implies $u\in c_{00}(-1/3)$.
\end{proof}

The above particular case looks pretty similar to the uniform setting, that
is the sequence space $c_{00}(\alpha)$, $\alpha\in\mathcal{E}_{ad}(\mathcal{A%
})$, is unique. We expose below the completely different situation of a
discrete evolution family admitting infinitely many such spaces.

\begin{example}
\label{ex.2} For the discrete evolution family $A(m,n)=e^{f(m)-f(n)}\mathrm{%
Id}$, where
\begin{equation*}
f(t)=-2\sqrt{t}\cos\sqrt{t}+2\sin\sqrt{t}-t,\;t\geq0,
\end{equation*}
the following properties hold:

\begin{enumerate}
\item[(i)] $\mathcal{E}_{ad}\left( \mathcal{A}\right) =(-1,\infty)$;

\item[(ii)] $\mathcal{A}$ is $\alpha$-uniform exponentially bounded if $%
\alpha\geq0$, thus $c_{00}(\alpha)=c_{00}$;

\item[(iii)] $\mathcal{A}$ is not $\alpha$-uniform exponentially bounded if $%
\alpha\in(-1,0)$;

\item[(iv)] For all $\alpha<\beta$ in $(-1,0)$, the inclusion $%
c_{00}(\alpha)\subset c_{00}(\beta)$ is strict.
\end{enumerate}
\end{example}

\begin{proof}
In what follows letters $t$, $s$ denote real variables, while letters $m$, $%
n $, $p$ only refer to non-negative integers. For any fixed $\alpha\in
\mathbb{R}$ we set
\begin{equation*}
E_{\alpha}(m,n)=f(m)-f(n)-\alpha(m-n)=f_{\alpha}(m)-f_{\alpha}(n), \;
(m,n)\in\Delta,
\end{equation*}
where
\begin{equation*}
f_{\alpha}(t)=f(t)-\alpha t=-2\sqrt{t}\cos\sqrt{t}+2\sin\sqrt{t}%
-(\alpha+1)\,t,\;t\geq0.
\end{equation*}

(i) The density of the set $\left\{ \sin m\left\vert m\in\mathbb{N}\right.
\right\} $ in $\left[ -1,1\right] $ implies that there exists an unbounded
sequence of non-negative integers $(m_{n})_{n\in\mathbb{N}}$ with $\underset{%
n\rightarrow\infty}{\lim}\sin\sqrt{m_{n}}=0$, and $\underset{%
n\rightarrow\infty}{\lim}\cos\sqrt{m_{n}}=-1$. For $\alpha=-1$ one has%
\begin{align*}
\lim\limits_{n\rightarrow\infty}E_{-1}(m_{n},0) & =\lim\limits_{n\rightarrow
\infty}\left( -2\sqrt{m_{n}}\cos\sqrt{m_{n}}+2\sin\sqrt{m_{n}}\right) \\
& =\lim\limits_{n\rightarrow\infty}\left( -2\sqrt{m_{n}}\cos\sqrt{m_{n}}%
\right) =\infty,
\end{align*}
which implies $-1\notin\mathcal{E}_{ad}\left( \mathcal{A}\right) $. If $%
\alpha>-1$, observing that $\lim\limits_{m\rightarrow\infty}E_{\alpha
}(m,n)=-\infty$, one gets $\sup\limits_{m\in\Delta_{n}}E_{%
\alpha}(m,n)=a_{n}<\infty$. We thereby conclude that $\mathcal{E}_{ad}\left(
\mathcal{A}\right) =(-1,\infty)$.

(ii) If $\alpha\geq0$, then $f_{\alpha}^{\prime}(t)=\sin\sqrt{t}%
-(\alpha+1)\leq0$, thus $E_{\alpha}(m,n)\leq0$. This proves that $\mathcal{A}
$ is $\alpha$-uniform exponentially bounded, and by Proposition \ref{P1} we
deduce that $c_{00}(\alpha)=c_{00}$.

(iii) If $\alpha\in(-1,0)$, then set $\theta=\arcsin\left( \alpha+1\right)
\in\left( 0,\frac{\pi}{2}\right) $. Equation $f_{\alpha}^{\prime}(t)=0$ has
two families of solutions:
\begin{equation*}
t_{n}^{\prime}=\left( 2n\pi+\theta\right) ^{2}\text{, }t_{n}^{\prime\prime
}=\left( \left( 2n+1\right) \pi-\theta\right) ^{2}\text{, }n\in \mathbb{N}.
\end{equation*}
Clearly $f_{\alpha}$ is an increasing map on each interval $(t_{n}^{\prime
},t_{n}^{\prime\prime})$, and decreasing on each interval $%
(t_{n}^{\prime\prime},t_{n+1}^{\prime})$. Since
\begin{equation*}
f_{\alpha}(t_{n+1}^{\prime\prime})-f_{\alpha}(t_{n}^{\prime\prime})=4\pi%
\lbrack\cos{\theta}-\sin{\theta}(2\pi-\theta+2n\pi)],
\end{equation*}
the sequence $\left( f_{\alpha}(t_{n}^{\prime\prime})\right) _{n\in \mathbb{N%
}}$ is decreasing for sufficiently large $n$, say $n\geq n_{0}(\alpha)$. Let
us denote $s_{n}=\left( 2n\pi+\frac{\pi}{2}\right)
^{2}\in(t_{n}^{\prime},t_{n}^{\prime\prime})$. It readily follows that for
any $n\geq n_{0}(\alpha)$ the following identity holds:%
\begin{equation*}
\sup\limits_{t\geq s_{n}}E_{\alpha}(t,s_{n})=E_{\alpha}(t_{n}^{\prime\prime
},s_{n})=f_{\alpha}(t_{n}^{\prime\prime})-f_{\alpha}(s_{n})=4n\pi
\varphi(\theta)+\psi(\theta),
\end{equation*}
where
\begin{equation*}
\varphi(\theta)=\cos{\theta}-\sin{\theta}\left( \frac{\pi}{2}-\theta\right)
,
\end{equation*}
and%
\begin{equation*}
\psi(\theta)=2(\pi-\theta)\cos{\theta}+2\sin{\theta}-\sin{\theta}\left(
\frac{\pi}{2}-\theta\right) \left( \frac{3\pi}{2}-\theta\right) -2.
\end{equation*}
As $\varphi$ is a decreasing map on $\left[ 0,\frac{\pi}{2}\right] $, then $%
\varphi(\theta)>\varphi(\frac{\pi}{2})=0$, and this implies $\lim
\limits_{n\rightarrow\infty}E_{\alpha}(t_{n}^{\prime\prime},s_{n})=\infty$.
The map $f_{\alpha}$ possesses an important property:
\begin{equation}
\left\vert t^{\prime}-t^{\prime\prime}\right\vert \leq1\Rightarrow\left\vert
f_{\alpha}(t^{\prime})-f_{\alpha}(t^{\prime\prime})\right\vert =\left\vert
\sin\sqrt{\xi}-(\alpha+1)\right\vert \leq2,   \label{eq3-e}
\end{equation}
with some $\xi$ between $t^{\prime}$ and $t^{\prime\prime}$. Set $p_{n}=%
\left[ s_{n}\right] +1$, and $q_{n}=\left[ t_{n}^{\prime\prime }\right] $.
From the identity%
\begin{equation*}
E_{\alpha}\left( q_{n},p_{n}\right)
=f_{\alpha}(q_{n})-f_{\alpha}(t_{n}^{\prime\prime})+f_{\alpha}(s_{n})-f_{%
\alpha}\left( p_{n}\right) +E_{\alpha}(t_{n}^{\prime\prime},s_{n}),
\end{equation*}
the monotonicity of the function $f_{\alpha}$ implies
\begin{equation*}
E_{\alpha}(t_{n}^{\prime\prime},s_{n})-E_{\alpha}\left( q_{n},p_{n}\right)
=f_{\alpha}(t_{n}^{\prime\prime})-f_{\alpha}(q_{n})+f_{\alpha}\left(
p_{n}\right) -f_{\alpha}(s_{n})\geq0,
\end{equation*}
and by \eqref{eq3-e} we obtain
\begin{equation}
0\leq E_{\alpha}(t_{n}^{\prime\prime},s_{n})-E_{\alpha}\left(
q_{n},p_{n}\right) \leq4.   \label{eq.1-e}
\end{equation}
Setting $\chi\left( \alpha,n\right) =\sup\limits_{m\in\Delta_{n}}E_{\alpha
}\left( m,n\right) $ one has%
\begin{equation*}
E_{\alpha}\left( q_{n},p_{n}\right) \leq\chi\left( \alpha,p_{n}\right) \leq%
\underset{t\geq p_{n}}{\sup}\ E_{\alpha}\left( t,p_{n}\right) \leq\underset{%
t\geq s_{n}}{\sup}\ E_{\alpha}\left( t,s_{n}\right)
=E_{\alpha}(t_{n}^{\prime\prime},s_{n}),
\end{equation*}
therefore inequality \eqref{eq.1-e} yields
\begin{equation}
0\leq E_{\alpha}(t_{n}^{\prime\prime},s_{n})-\chi\left( \alpha,p_{n}\right)
\leq4.   \label{eq2-e}
\end{equation}
It follows now that $\lim\limits_{n\rightarrow\infty}E_{\alpha}(t_{n}^{%
\prime\prime},s_{n})=\infty$ implies $\lim\limits_{n\rightarrow\infty}\chi%
\left( \alpha,p_{n}\right) =\infty$, hence (iii) holds.

(iv) For any fixed $-1<\alpha<\beta<0$ we set
\begin{equation*}
\theta=\arcsin(\alpha+1)\text{, }\zeta=\arcsin(\beta+1)\in\left( 0,\frac{\pi
}{2}\right) .
\end{equation*}
Let us define the real maps
\begin{equation*}
g_{\alpha}\left( s\right) =\underset{t\geq s}{\sup}E_{\alpha}\left(
t,s\right) ,\;g_{\beta}\left( s\right) =\underset{t\geq s}{\sup}E_{\beta
}\left( t,s\right) ,\text{ and }g\left( s\right) =g_{\alpha}\left( s\right)
-g_{\beta}\left( s\right) ,\;s\geq0.
\end{equation*}
Identity
\begin{align*}
g(s_{n}) & =\sup\limits_{t\geq p_{n}}E_{\alpha}(t,s_{n})-\sup\limits_{t\geq
p_{n}}E_{\beta}(t,s_{n})=E_{\alpha}(t_{n}^{\prime\prime},s_{n})-E_{\beta
}(t_{n}^{\prime\prime},s_{n}) \\
& =4n\pi\left( \varphi(\theta)-\varphi(\zeta)\right) +\left( \psi
(\theta)-\psi(\zeta)\right)
\end{align*}
implies $\underset{n\rightarrow\infty}{\lim}g(s_{n})=\infty$. Let us denote $%
\lambda_{n}=\chi\left( \alpha,p_{n}\right) -\chi\left( \beta,p_{n}\right) $.
Inequality \eqref{eq2-e} yields%
\begin{equation*}
\left\vert g(s_{n})-\lambda_{n}\right\vert \leq E_{\alpha}(t_{n}^{\prime
\prime},s_{n})-\chi\left( \alpha,p_{n}\right) +E_{\beta}(t_{n}^{\prime
\prime},s_{n})-\chi\left( \beta,p_{n}\right) \leq8,
\end{equation*}
and $\underset{n\rightarrow\infty}{\lim}g(s_{n})=\infty$ implies $\underset{%
n\rightarrow\infty}{\lim}\lambda_{n}=\infty$. For each $n\in\mathbb{N}$ and $%
\alpha_{0}\in\mathcal{E}_{ad}\left( \mathcal{A}\right) $\ we put $\mu\left(
\alpha_{0},n\right) =\max\left\{ \chi\left( \alpha_{0},n\right) ,\chi\left(
\alpha_{0},p_{n}\right) \right\} $. Let us define the sequence $u=\left(
u_{n}\right) _{n\in\mathbb{N}}$ by
\begin{equation*}
u_{n}=e^{-\frac{1}{2}\lambda_{n}-\mu\left( \beta,n\right) },\;n\in \mathbb{N}%
.
\end{equation*}
Since
\begin{equation*}
\Vert u_{n}\Vert_{n,\beta}=\sup\limits_{m\in\Delta_{n}}e^{-\beta\left(
m-n\right) }\left\Vert A\left( m,n\right) u_{n}\right\Vert =e^{\chi\left(
\beta,n\right) -\frac{1}{2}\lambda_{n}-\mu\left( \beta,n\right) }\leq e^{-%
\frac{1}{2}\lambda_{n}}\rightarrow0,
\end{equation*}
we get $u\in c_{00}(\beta)$, and from
\begin{align*}
\Vert u_{p_{n}}\Vert_{p_{n},\alpha} & =\underset{m\in\Delta_{p_{n}}}{\sup }%
e^{-\alpha\left( m-p_{n}\right) }\left\Vert A\left( m,p_{n}\right)
u_{p_{n}}\right\Vert =e^{\chi\left( \alpha,p_{n}\right) -\frac{1}{2}%
\lambda_{n}-\mu\left( \beta,p_{n}\right) } \\
& =e^{\chi\left( \alpha,p_{n}\right) -\frac{1}{2}\lambda_{n}-\chi\left(
\beta,p_{n}\right) }=e^{\frac{1}{2}\lambda_{n}}\rightarrow\infty,
\end{align*}
we obtain $u\notin c_{00}(\alpha)$, which ends the proof.
\end{proof}

\section*{A final comment}

We anticipate that the techniques and constructions introduced in the
present paper should lead to new characterizations of nonuniform
hyperbolicity of discrete dynamics, following for example the path in \cite%
{Mi.Ra.Sc.1998}.

To act likewise, it is quite necessary to work on the new Banach function
space
\begin{equation*}
c_{0}(\alpha)=\left\{ u=(u_{n})_{n\in\mathbb{N}}\subset E\text{: }%
\lim\limits_{n\rightarrow\infty}\Vert u_{n}\Vert_{n,\alpha}=0\right\}
\end{equation*}
instead of $c_{00}(\alpha)$, and we think that a conclusion as in Theorem
4.3 from \cite{Mi.Ra.Sc.1998} is expected.

\section*{Acknowledgments} 

 We would like to thank to the referees for their useful comments and suggestions meant to improve the quality of our paper.

\appendix

\section{A convergence result}

\label{ap1}

The following lemma gives a convergence result for real series.

\begin{lemma}
\label{l1.A} If $(a_{n})_{n\in\mathbb{N}}$ and $(b_{n})_{n\in\mathbb{N}}$
are two positive real sequences such that the series $\sum\limits_{n=0}^{%
\infty }a_{n}$ converges and $\lim\limits_{n\rightarrow\infty}b_{n}=0$, then
\begin{equation*}
\lim\limits_{n\rightarrow\infty}\sum\limits_{k=0}^{n}a_{n-k}b_{k}=0.
\end{equation*}
\end{lemma}

\begin{proof}
For each $n\in\mathbb{N}$, set $x_{n}=\sum\limits_{k=0}^{n}a_{n-k}b_{k}$, $%
S_{n}=\sum\limits_{k=0}^{n}a_{k}$, and let $S=\sum\limits_{k=0}^{\infty}a_{k}
$. As $\lim\limits_{n\rightarrow\infty}b_{n}=0$, then $b_{n}\leq K$ for some
$K>0.$ For $\varepsilon>0$ there exits $n_{1}=n_{1}(\varepsilon )\in\mathbb{N%
}$ such that
\begin{equation*}
\label{eq.l2}b_{n}<\dfrac{\varepsilon}{2S},\text{ for all }n\geq n_{1}.
\end{equation*}
The convergence of the series implies the existence of some $%
n_{2}=n_{2}(\varepsilon)\in\mathbb{N}$ such that for all $n\geq n_{2}$ and $%
p\in\mathbb{N}^{\ast}$ one has
\begin{equation*}
\label{eq.l3}S_{n+p}-S_{n}<\dfrac{\varepsilon}{2K}.
\end{equation*}
Let $n_{0}=\max\{n_{1},n_{2}\}$ and choose $n\geq2n_{0}$. Estimation
\begin{align*}
x_{n} & =\left( a_{n}b_{0}+\cdots+a_{n-(n_{0}-1)}b_{n_{0}-1}\right) +\left(
a_{n-n_{0}}b_{n_{0}}+\cdots+a_{0}b_{n}\right) \\
& \leq K\left( a_{n-n_{0}+1}+\cdots+a_{n}\right) +\dfrac{\varepsilon}{2S}%
\left( a_{0}+\cdots+a_{n-n_{0}}\right) \\
& =K\left( S_{n}-S_{n-n_{0}}\right) +\dfrac{\varepsilon}{2S}\sum
\limits_{k=0}^{n-n_{0}}a_{k} \\
& <\dfrac{\varepsilon}{2}+\dfrac{\varepsilon}{2}=\varepsilon
\end{align*}
shows that $x_{n}\rightarrow0$.
\end{proof}

\section{A consequence of Faulhaber's formula}

\label{ap.2}

Faulhaber's formula expresses the sum of the $p$th powers of the first $n$
positive integers as a $(p+1)$th-degree polynomial function of $n$.
Precisely, for all $n\in\mathbb{N}^{\ast}$ and $p\in\mathbb{N}$, the
following identity holds:
\begin{equation*}
\label{eq.F}\sum\limits_{k=1}^{n}k^{p}=\dfrac{1}{p+1}\sum%
\limits_{j=0}^{p}(-1)^{j}%
\begin{pmatrix}
p+1 \\
j \\
\end{pmatrix}
B_{j}n^{p+1-j},
\end{equation*}
where $B_{j}$ denotes the $j$th Bernoulli number with the convention $%
B_{1}=-1/2$. Since $B_{0}=1$, one has
\begin{equation}
\sum\limits_{k=1}^{n}k^{p}\geq\dfrac{n^{p+1}}{p+1},\;n\in\mathbb{N}^{\ast
},\;p\in\mathbb{N}.   \label{eq.c}
\end{equation}

\end{document}